\documentclass[11pt, reqno]{amsart}
\usepackage{amsfonts}
\usepackage{amssymb}
\usepackage{latexsym}
\usepackage[final]{hyperref}
\usepackage[]{enumerate}
\usepackage{verbatim}

\usepackage{graphicx}       
\usepackage[utf8]{inputenc} 
\usepackage{amsmath}

\usepackage{tikz}
\usepackage{cmap}
\usepackage{ulem}
\usepackage{soul}
\usepackage{comment}
\usepackage{fancybox}
\usepackage{relsize}
\usepackage{pifont}

\newcommand{\norm}[1]{\left\Vert #1 \right\Vert}

\DeclareMathOperator{\sech}{sech}

\newtheorem{theorem}{Theorem}

\newtheorem{lemma}{Lemma}
\newtheorem{corollary}{Corollary}
\newtheorem{assumption}{Assumption}

\theoremstyle{definition}

\theoremstyle{remark}

\title
[
	Travelling waves
]
{
	Travelling waves in the
	Boussinesq type systems
}

\author[Dinvay]{Evgueni Dinvay}

\email{ Evgueni.Dinvay@inria.fr }
 
\address
{
	Inria Rennes - Bretagne Atlantique
	\\
	Campus universitaire de Beaulieu Avenue du G\'en\'eral Leclerc
	\\
	35042 Rennes Cedex
	\\
	France
}

\date{\today}



\begin{document}

\begin{abstract}
Considered herein are a number of variants of the Boussinesq type systems
modelling surface water waves.
Such equations were derived by different authors
to describe the two-way propagation of long gravity waves.
A question of existence of special solutions,
the so called solitary waves, is of a particular interest.
There are a number of studies relying on a variational approach
and a concentration-compactness argument.
These proofs are technically very demanding and may vary
significantly from one system to another.
Our approach is based on the implicit function theorem,
which makes the treatment easier and more unified.

%
%
\end{abstract}

\maketitle
\section{Introduction}
\setcounter{equation}{0}

Consideration is given to the following one-dimensional
Boussinesq-type system
\begin{equation}
\label{Regularised_sys}
\left\{
\begin{aligned}
	K_b \partial_t \eta 
	+  K_a \partial_x v + \partial_x (\eta v)
	&= 0
	, \\
	K_d \partial_t v
	+ K_c \partial_x \eta
	+ \partial_x v^2 / 2
	&= 0
	,
\end{aligned}
\right.
\end{equation} 
where $ K_a, \ldots, K_d$ are Fourier multiplier operators in
the space of tempered distributions $\mathcal S'(\mathbb R)$.
The space variable is $x \in \mathbb R$ and the time variable
is $t \in \mathbb R$.
The unknowns $\eta$, $v$ are real valued functions of these
variables,
representing the free surface elevation
and velocity at a different level in the fluid layer, respectively.
A classical example is the so called $(a,b,c,d)$-system with
\(
	K_a = 1 + a \partial_x^2
	,
\)
\(
	K_b = 1 - b \partial_x^2
	,
\)
\(
	K_c = 1 + c \partial_x^2
	,
\)
\(
	K_d = 1 - d \partial_x^2
	,
\)
introduced in \cite{Bona_Chen_Saut1}.
Other examples with non-local operators will be given in Conclusion.

Abusing slightly notations and exploiting the travelling wave ansatz
$\eta(x,t) = \eta(x - \omega t)$,
$v(x,t) = v(x - \omega t)$, with $\omega \in \mathbb R$,
we rewrite System \eqref{Regularised_sys} in the form
\begin{equation}
\label{travelling_sys}
\left\{
\begin{aligned}
	- \omega K_b \eta 
	+  K_a v + \eta v
	&= 0
	, \\
	- \omega K_d v
	+ K_c \eta
	+ v^2 / 2
	&= 0
	.
\end{aligned}
\right.
\end{equation} 
Note that we assume $\eta$, $v$ vanishing at infinity.
If $v$ is known then obviously  $\eta$ can be found  through the second equation as
\[
	\eta = \omega  K_c^{-1} K_d v
	- K_c^{-1} v^2 / 2
	.
\]
Thus the problem reduces to investigation of a single equation on $v$
of the form
\begin{equation}
\label{solitary_equation}
	\left( \omega^2 - M^2 \right) v
	=
	\omega F v^2
	+
	\omega  G( v H v )
	+
	T(v,v,v)
	,
\end{equation}
where $F, G, H$ are Fourier multipliers
and $T$ is a trilinear bounded operator in some Sobolev space
$H^s_e(\mathbb R)$ of even functions.
Assigning, in particular,
$ M = \sqrt{ K_a K_b^{-1} K_c K_d^{-1} } $,
$F = K_d^{-1} / 2$,
$G = K_b^{-1} K_c K_d^{-1}$,
$H = K_c^{-1} K_d$
and
$T(v,v,v) = K_b^{-1} K_c K_d^{-1} \left( v K_c^{-1} v^2 / 2 \right)$,
we arrive back to the initial problem \eqref{travelling_sys}.
The general equation \eqref{solitary_equation}
is of main interest below.

Different particular versions 
of System \eqref{Regularised_sys}
appeared in
\cite{
	Aceves_Sanchez_Minzoni_Panayotaros,
	Bona_Chen_Saut1,
	Dinvay_Dutykh_Kalisch,
	Emerald2020,
	Hur_Pandey}.
For some of them existence of solitary wave solutions
was proved in
\cite{
	Chen_Nguyen_Sun2010,
	Chen_Nguyen_Sun2011,
	Dinvay_Nilsson,
	Nilsson_Wang}.
These results are obtained by
reformulation of \eqref{travelling_sys}
as a variational problem, different in each case,
and by appealing to the concentration-compactness principle.
The corresponding proofs are technically very demanding and vary
significantly from one particular system to another.
Another shortcoming of this approach is that uniqueness
of minimizers needs in general additional studies.
On the other hand such treatment can sometimes be advantageous,
as for example in \cite{Chen_Nguyen_Sun2010},
where also stability for the set of minimizers was obtained.

The variational approach has been extensively used to prove
existence of solitary wave solutions to
the single unidirectional travelling wave equation 
\begin{equation}
\label{Whitham}
	\omega v = Mv + n(v)
\end{equation}
as well, see
\cite{Ehrnstrom_Groves_Wahlen, Weinstein1987},
for instance.
An alternative elegant proof
with an entirely different
approach was given by Stefanov and Wright
\cite{Stefanov_Wright2018}.
They have rescaled the travelling wave equation \eqref{Whitham}
introducing a small parameter that lead in the limit
to the KdV travelling wave equation.
Now existence together with uniqueness come from appeal
to the implicit function theorem.
In the current work we extend their approach
to Equation \eqref{solitary_equation}.
Note that it provides us with a unified treatment
for all systems \eqref{Regularised_sys},
including those of them that have not been proved
possessing solitary waves so far.
We regard concrete examples in Conclusion.
One can also add, that
the implicit function argument normally gives small solutions,
whereas the variational method may produce big ones as well.

\section{Notations, Assumptions and Main result}
\setcounter{equation}{0}

We start this section by recalling
all the necessary standard notations.
For any positive numbers $a$ and $b$ we write
$a \lesssim b$ if there exists a constant $C$ independent
of $a, b$ such that $a \leqslant Cb$.
The Fourier transform is defined by the usual formula
\[
	\widehat{f}(\xi) = \mathcal F(f)(\xi) =
	\int	 f(x) e^{-i\xi x} dx
\]
on Schwartz functions.
By the Fourier multiplier operator $\varphi(D)$
with symbol $\varphi$ we mean
the line
\(
	\mathcal F \left( \varphi(D) f \right)
	=
	\varphi(\xi) \widehat{f}(\xi)
	.
\)
In particular, $D = -i\partial_x$ is
the Fourier multiplier associated with the symbol
$\varphi(\xi) = \xi$.
For any $\alpha \in \mathbb R$
the Bessel potential of order $-\alpha$ is
the Fourier operator $J^{\alpha} = \langle D \rangle ^{\alpha}$,
where we exploit the symbol notation
$J(\xi) = \langle \xi \rangle = \sqrt{1 + \xi^2}$.
Note that below we stick to this abuse of notations,
namely we notate Fourier multipliers and their symbols
by the same letters, for example,
$F$ appeared as an operator in \eqref{solitary_equation}
can stand for the corresponding symbol $F(\xi)$ in other context.
The $L^2$-based Sobolev space $H^{\alpha} (\mathbb R)$ is defined
by the norm
\(
	\norm{f} _{H^{\alpha}}
	=
	\norm{J^{\alpha}f} _{L^2}
	,
\)
whereas
$H_e^{\alpha} (\mathbb R)$
is its subspace of even functions.

In order to study Equation \eqref{solitary_equation}
we will need to rescale it by a scalar parameter
$\varepsilon \in \mathbb R \setminus \{ 0 \}$.
For any function $\varphi : \mathbb R \to \mathbb R$
we use notation
\(
	\varphi_{\varepsilon}(x)
	=
	\varphi (\varepsilon x)
	.
\)
For a Fourier multiplier $F$ by $F_{\varepsilon}$
we mean the operator with the symbol
$F_{\varepsilon} (\xi) = F(\varepsilon \xi)$.
Notation $F_0$ stands for the constant $F(0)$.
In general, for any $n$-linear operator $B$
we define operator $B_{\varepsilon}$ by the line
\[
	B_{\varepsilon}(f^1, \ldots, f^n)
	=
	\left( B(f^1_{\varepsilon}, \ldots, f^n_{\varepsilon}) \right)
	_{ \frac 1  {\varepsilon} }
	,
	\quad
	\varepsilon \ne 0
	.
\]
We deliberately use the same notation for rescaled
general operators and Fourier multipliers,
since obviously
rescaling of a Fourier multiplier (by rescaling the symbol)
coincides with the composition of rescaling on the physical side
by $\varepsilon$ and $1/\varepsilon$.

Solvability of \eqref{solitary_equation} is investigated
by approximation with the Korteweg–de Vries equation
\begin{equation}
\label{KdV}
	- \partial_x^2 v + v - \gamma v^2 = 0
	,
\end{equation}
where $\gamma \ne 0$ is to be given below.
It has a unique solution
\(
	v = \sigma
\)
with
\begin{equation}
\label{sigma_definition}
	\sigma(x) =
	\frac 3{2\gamma} \sech^2 \left( \frac x2 \right)
	.
\end{equation}

\begin{assumption}
\label{main_assumption}
	The multiplier symbols
	\(
		M, F, G, H : \mathbb R \to \mathbb R
	\)
	are even and  there exist $\xi_1 > 0$ and $\beta < 1$
	such that
	\begin{enumerate}
		\item
		\label{M_assumption_1}
		\(
			M \in C^{3,1} ([ - \xi_1, \xi_1 ])
		\)
		that is, it has Lipschitz continuous third derivative,
		with
		\begin{equation}
		\label{M_assumption_inequality_1}
			m_1 = \sup _{ |\xi| \geqslant \xi_1 } M(\xi) < M(0)
			,
		\end{equation}
		\begin{equation}
		\label{M_assumption_inequality_2}
			m_2 = \sup _{ |\xi| \leqslant \xi_1 } M''(\xi) < 0
			;
		\end{equation}
		\item
		$M$ is bounded from below as
		\begin{equation}
			m_3 = \inf _{ \xi \in \mathbb R } ( M(\xi) + M(0) ) > 0
			;
		\end{equation}
		\item
		there exist derivatives $M', F', G', H'$ on $\mathbb R$
		satisfying
		\begin{equation}
			\left| M'(\xi) \right|
			,
			\left| F'(\xi) \right|
			,
			\left| G'(\xi) \right|
			,
			\left| H'(\xi) \right|
			\lesssim \langle \xi \rangle ^{\beta}
			;
		\end{equation}
		\item
		there exist bounded derivatives $F'', G'', H''$
		on $[- \xi_1, \xi_1]$;
		\item
		the multiplier symbol $F$ is bounded on $\mathbb R$;
		\item
		$G$ is smoothing $H$, that is there exists $s_h \geqslant 0$
		satisfying
		\[
			|G(\xi)| \lesssim \langle \xi \rangle ^{-s_h}
			, \quad
			|H(\xi)| \lesssim \langle \xi \rangle ^{s_h}
			;
		\]
		\item
		there is a relation between $F_0, G_0$ and $H_0$
		of the form
		\begin{equation}
		\label{FGH_assumption_inequality}
			\gamma =
			- \frac 1{ M''(0) }
			\left[
				F(0) + G(0)H(0)
			\right]
			\ne 0
			.
		\end{equation}
	\end{enumerate}
\end{assumption}

Operators $M, F, G, H$ are defined in subspaces of even functions,
due to evenness of the corresponding symbols.
One can easily check that there is no loss of generality
in assuming that $\xi_1 > 0$ and $\beta < 1$ are unique
for all symbols $M, F, G, H$,
since one can always take the minimum such $\xi_1$
and the maximum such $\beta$.
Clearly, $s_h \leqslant 1 + \beta < 2$.
Note that for $s > 1/2$
and $s \geqslant s_h / 2$
the bilinear part in Equation \eqref{solitary_equation}
is bounded in $H^s(\mathbb R) \times H^s(\mathbb R)$,
as will be shown below.

\begin{assumption}
\label{T_assumption}
	Let $s \in \mathbb R$ be such that
	$T$ is a trilinear bounded operator in $H^s_e(\mathbb R)$
	and for some $s_t < 2$ it holds true that
	\[
		\norm{ T_{\varepsilon} } \lesssim | \varepsilon |^{-s_t}
		, \quad
		\norm{ T_{\varepsilon}(\sigma, \sigma, \sigma) }
		_{H^s}
		= o \left( \frac 1{\varepsilon} \right)
	\]
	as $\varepsilon \to 0$.
	Here $\sigma$ is defined in \eqref{sigma_definition}.
\end{assumption}

Both inequalities are easy to check for the concrete
examples regarded in the last section.
Note that
\(
		\norm{ T_{\varepsilon} } \leqslant | \varepsilon |^{ - s - 1 }
		\norm{ T }
\)
for $0 < |\varepsilon| \leqslant 1$.
The first inequality is guaranteed automatically  provided $s < 1$,
and the second one provided $s < 0$.

In the next theorem we prove existence of small
travelling waves corresponding to velocities
slightly above $M(0)$.
We introduce the wave speed parametrised by $\varepsilon \in \mathbb R$
as follows
\begin{equation}
\label{omega_definition}
	\omega_{\varepsilon} = M(0) - \frac 12 M''(0) \varepsilon^2
	.
\end{equation}

\begin{theorem}
\label{main_theorem}
	Let Assumptions \ref{main_assumption} and \ref{T_assumption}
	be met with
	$s \geqslant 1$.
	Then there exists $\varepsilon_0 > 0$,
	so that for every
	\(
		\varepsilon \in (0, \varepsilon_0)
		,
	\)
	there is a unique travelling wave
	solution to Equation \eqref{solitary_equation},
	associated with the phase velocity
	\(
		\omega = \omega_{\varepsilon}
	\)
	given by \eqref{omega_definition},
	of the form
	\(
		v(x) = \varepsilon^2 V^{\varepsilon} (\varepsilon x)
	\)
	with $V^{\varepsilon} \in H_e^s(\mathbb R)$ satisfying
	\[
		V^{\varepsilon}
		=
		\sigma + o_{H^s} ( \varepsilon )
		.
	\]
\end{theorem}

\section{Preliminaries}
\setcounter{equation}{0}

We  state an estimate, firstly appeared
in \cite{Klainerman_Selberg} in a weaker form,
and later sharpened in \cite{Selberg_Tesfahun}.

\begin{lemma}
\label{Selberg_Tesfahun_lemma}
	Suppose $a, b, c \in \mathbb R$.
	Then for any $f \in H^a(\mathbb R)$,
	$g \in H^b(\mathbb R)$ and $h \in H^c(\mathbb R)$
	the following inequality holds
	\begin{equation}
	\label{Selberg_Tesfahun}
		\lVert fgh \rVert _{L^1}
		\lesssim
		\lVert f \rVert _{H^a}
		\lVert g \rVert _{H^b}
		\lVert h \rVert _{H^c}
	\end{equation}
	provided that
	\[
		a + b + c > \frac 12
		,
	\]
	\[
		a + b \geqslant 0
		,
		\quad
		a + c \geqslant 0
		,
		\quad
		b + c \geqslant 0
		.
	\]
\end{lemma}

An immediate corollary of this Lemma
and duality argument is the following.

\begin{corollary}
\label{Selberg_Tesfahun_corollary_1}
	Let $s \geqslant 1$ and $f \in H^s(\mathbb R)$.
	Then
	\(
		A_f = J^{-2} f J^2
	\)
	is a bounded linear operator on $H^s(\mathbb R)$
	with the norm
	\(
		\norm{ A_f } \lesssim \norm{f} _{H^s}
		.
	\)
\end{corollary}

\begin{corollary}
\label{Selberg_Tesfahun_corollary_2}
	Given Assumption \ref{main_assumption}
	and $s > 1/2$ such that $s \geqslant s_h/2$,
	we have the following
	\[
		\norm{ G_{\varepsilon} ( v H_{\varepsilon} w ) }_{H^s}
		\lesssim
		|\varepsilon|^{-s_h}
		\norm{v}_{H^s} \norm{w}_{H^s}
		,
	\]
	where $0 < |\varepsilon| \leqslant 1$ and
	$v, w \in H^s(\mathbb R)$.
\end{corollary}

\begin{proof}
Indeed,
\[
	\norm{ G_{\varepsilon} ( v H_{\varepsilon} w ) }_{H^s}
	=
	\sup_{ \norm{u}_{H^{-s}} \leqslant 1 }
	\left| \int v ( H_{\varepsilon} w ) G_{\varepsilon} u \right|	
	\lesssim
	\sup_{ \norm{u}_{H^{-s}} \leqslant 1 }
	\norm{v}_{H^s}
	\norm{H_{\varepsilon} w}_{H^{s - s_h}}
	\norm{G_{\varepsilon} u}_{H^{s_h - s}}
	,
\]
where
\(
	\norm{H_{\varepsilon} w}_{H^{s - s_h}}
	\lesssim
	\norm{w}_{H^s}
\)
and
\(
	\norm{G_{\varepsilon} u}_{H^{s_h - s}}
	\lesssim
	|\varepsilon|^{-s_h}
	\norm{u}_{H^{-s}}
	,
\)
following from the inequalities
\(
	J(\varepsilon \xi) / J(\xi) \leqslant 1
\)
and
\(
	J(\xi) / J(\varepsilon \xi) \leqslant 1 / |\varepsilon|
\)
holding true for any $\xi \in \mathbb R$.
\end{proof}

\begin{lemma}
\label{main_lemma}
	Let $\varphi$ be a real-valued even function defined on $\mathbb R$
	such that
	\begin{enumerate}
		\item
		it has derivative $\varphi'$
		satisfying
		\(
			|\varphi'(\xi)| \lesssim \langle \xi \rangle ^{\beta}
		\)
		for some $\beta < 1$ and every $\xi \in \mathbb R$;
		\item
		it has a bounded second derivative $\varphi''$
		on some interval $(-\xi_1, \xi_1)$ around zero.
	\end{enumerate}
	Then
	\[
		\norm{ J^{-2} ( \varphi( \varepsilon D ) - \varphi(0) ) f } _{H^s}
		\lesssim
		| \varepsilon |^{ \frac {3 - 2\beta}{2 - \beta} }
		\norm{ f } _{H^s}
	\]
	for any $f \in H^s(\mathbb R)$ provided $|\varepsilon| \leqslant 1$.
\end{lemma}

\begin{proof}
Let $R > 0$ be large to be chosen below
dependent on $\varepsilon$.
One can assume that $\varepsilon > 0$ in addition.
We calculate the square of norm by integrating separately
over the ball $B_R = (-R, R)$
and its complement $B_R^c$ as follows
\begin{multline*}
	N_f^2 =	
	\norm{
		J^{-2} \frac{ \varphi( \varepsilon D ) - \varphi(0) }
		{\varepsilon} f
	} _{H^s}^2
	=
	\int \left| \varphi'(\theta_{\varepsilon \xi}) \right|^2
	\left| \widehat{f}(\xi) \right|^2
	\langle \xi \rangle ^{2s - 4} |\xi|^2 d\xi
	\\
	\leqslant
	\sup _{\varepsilon B_R} \left| \varphi' \right|^2
	\int _{B_R} \left| \widehat{f}(\xi) \right|^2
	\langle \xi \rangle ^{2s - 2} d\xi
	+
	C \int _{B_R^c} \langle \theta_{\varepsilon \xi} \rangle ^{2\beta}
	\left| \widehat{f}(\xi) \right|^2
	\langle \xi \rangle ^{2s - 2} d\xi
	.
\end{multline*}
Clearly, one can take $R$ in a way so that
\(
	\varepsilon R \leqslant \xi_1
	,
\)
and so
\[
	\sup _{\varepsilon B_R} \left| \varphi' \right|
	=
	\sup _{|\xi| \leqslant \varepsilon R}
	\left| \varphi'(\xi) - \varphi'(0) \right|
	\leqslant
	\varepsilon R \norm{\varphi''} _{L^{\infty}(-\xi_1, \xi_1)}
	.
\]
To estimate the second integral notice that for
$|\xi| \geqslant R$ we have
\[
	\langle \theta_{\varepsilon \xi} \rangle ^{2\beta}
	\langle \xi \rangle ^{-2}
	\leqslant
	\langle \varepsilon \xi \rangle ^{2\beta}
	\langle \xi \rangle ^{-2}
	\leqslant
	\langle \xi \rangle ^{2\beta - 2}
	\leqslant
	R ^{2\beta - 2}
	.
\]
Thus
\(
	N_f \lesssim
	\left( \varepsilon R + R^{\beta - 1} \right)
	\norm{f}_{H^s}
	,
\)
and so it is left to assign
\(
	R = \varepsilon^{-1 / (2 - \beta)} \xi_1
\)
in order to conclude with the proof.
\end{proof}
This Lemma is used below to approximate the operators
$F_{\varepsilon}, G_{\varepsilon}$
and $H_{\varepsilon}$,
whereas the operator $M_{\varepsilon}$ is approximated
with the help of the following two results.
\begin{corollary}
\label{main_corollary}
	Given Assumption \ref{main_assumption},
	for any $f \in H^s(\mathbb R)$ and
	$|\varepsilon| \leqslant 1$ it holds true that
	\[
		\norm{
			J^{-2}
			\left(
				\frac 1{ \omega_{\varepsilon} + M_{\varepsilon} }
				-
				\frac 1{2M_0}
			\right) f
		} _{H^s}
		\lesssim
		| \varepsilon |^{ \frac {3 - 2\beta}{2 - \beta} }
		\norm{ f } _{H^s}
	\]
\end{corollary}

\begin{proof}
For any $\xi \in \mathbb R$ we have
\[
	\left|
		\frac 1{ \omega_{\varepsilon} + M(\varepsilon \xi) }
		-
		\frac 1{2M_0}
	\right|
	=
	\left|
		\frac
		{ \omega_{\varepsilon} - M_0 + M(\varepsilon \xi) - M_0 }
		{ 2M_0 ( \omega_{\varepsilon} + M(\varepsilon \xi) ) }
	\right|
	\leqslant
	\frac 1{2|M_0 m_3|}
	\left(
		\frac 12 |M''(0)| \varepsilon^2
		+ | M(\varepsilon \xi) - M_0 |
	\right)
	,
\]
which after implementation of Lemma \ref{main_lemma}
finishes the proof.
\end{proof}

\begin{lemma}
\label{Stefanov_Wright_lemma_1}
	Let $M$ satisfy
	Assumption \ref{main_assumption}.
	Then there is $C > 0$ such that for any $\varepsilon \neq 0$
	it holds that
	\[
		\sup _{ \xi \in \mathbb R }
		\left|
			\frac
			{ \varepsilon^2 }
			{ M(0) - \frac 12 M''(0) \varepsilon^2 - M( \varepsilon \xi ) }
			+
			\frac 2{ M''(0) ( 1 + \xi^2 ) }
		\right|
		\leqslant
		C \varepsilon^2
		.
	\]
\end{lemma}

A proof is given in \cite{Stefanov_Wright2018}.
This lemma implies that there exists the inverse bounded operator
\begin{equation}
\label{Stefanov_Wright_inverse}
	\varepsilon^2 \left( \omega_{\varepsilon} - M_{\varepsilon} \right)^{-1}
	=
	- \frac 2{ M''(0) } \left( 1 - \partial_x^2 \right)^{-1}
	+ O_{\mathcal B(H^s)}(\varepsilon^2)
\end{equation}
as $\varepsilon \to 0$.
We provide the following lemma with a complete proof,
though the idea can be found
in \cite{Friesecke_Pego1999, Stefanov_Wright2018},
for example.

\begin{lemma}
\label{Stefanov_Wright_lemma_2}
	Let $s \geqslant 0$, $\gamma \ne 0$
	and $\sigma(x)$ be defined by \eqref{sigma_definition}.
	Then the operator
	\begin{equation}
	\label{K_definition}
		\mathcal K
		=
		1 - 2\gamma J^{-2} (\sigma \cdot)
	\end{equation}
	is bounded and has a bounded inverse
	in $H_e^s(\mathbb R)$.
\end{lemma}

\begin{proof}
The fact that
\(
		\mathcal K \in \mathcal B( H_e^s(\mathbb R) )
\)
is obvious.
By the bounded inverse theorem it is enough
to prove that $\mathcal K$ is invertible.
Firstly, we regard the case $0 \leqslant s < 3/2$.
The composition
\(
	\langle \xi \rangle^{-2} \mathcal F \sigma
\)
is a Hilbert-Schmidt operator from $L^2(\mathbb R, dx)$
to $L^2 \left( \mathbb R, \langle \xi \rangle^{2s} d\xi \right)$.
Then from continuity of the injection
\(
	H^s(\mathbb R) \hookrightarrow L^2(\mathbb R)
\)
and of the inverse Fourier transform
\(
	\mathcal F^{-1} :
	L^2 \left( \mathbb R, \langle \xi \rangle^{2s} d\xi \right)
	\to H^s(\mathbb R)
\)
we deduce that $\mathcal K - 1$ is compact
in $H^s(\mathbb R)$.
Now let us assume that $\mathcal K$ is not invertible
in the subspace of even functions.
Then by the Fredholm alternative
there is a non-trivial
\(
	f_0 \in \ker \mathcal K
	\subset H_e^s(\mathbb R)
	,
\)
and so
\(
	f_0 = 2\gamma J^{-2} (\sigma f_0)
	.
\)
It implies that
\(
	f_0 \in H_e^2(\mathbb R)
\)
and
\(
	\left( - \partial_x^2 + 1 - 2\gamma \sigma \right) f_0 = 0
	.
\)
However, it is known about the last operator
that its kernel is spanned by $\sigma'$.
This implies that $f_0$ is odd, which is a contradiction.
Thus $\mathcal K$ is invertible.

Finally, regarding the case $s \geqslant 3/2$,
for any $g \in H_e^s(\mathbb R)$
one can find a unique $f \in H_e^1(\mathbb R)$
such that $\mathcal K f = g$.
Hence
\(
	f = 2\gamma J^{-2} (\sigma f) + g
\)
implying
$f \in H_e^{ \min (3, s) }(\mathbb R)$.
After several repetitions we arrive to
$f \in H_e^s(\mathbb R)$,
which proves invertibility of $\mathcal K$.
\end{proof}

\section{Proof}
\setcounter{equation}{0}

We start by introducing a mapping from
$H^s_e(\mathbb R) \times \mathbb R$ to $H^s_e(\mathbb R)$
as follows.
For $\varepsilon \ne 0$ we define
\begin{equation}
\label{Phi_definition}
	\Phi(v, \varepsilon) =
	v - \varepsilon^2 \left( \omega_{\varepsilon}^2 - M_{\varepsilon}^2 \right)^{-1}
	\left[
		\omega_{\varepsilon} F_{\varepsilon} v^2
		+
		\omega_{\varepsilon}  G_{\varepsilon} ( v H_{\varepsilon} v )
		+
		\varepsilon^2 T_{\varepsilon}(v,v,v)
	\right]
\end{equation}
and $\Phi(v, 0)$ is defined as a limit when $\varepsilon \to 0$ provided it exists.
Firstly, we analyse $\Phi$ around the hyperplane  $\varepsilon = 0$
and prove that $\Phi$ is continuous at $(\sigma, 0)$ with $\Phi(\sigma, 0) = 0$.
In order to simplify the presentation we restrict the domain of $\Phi$
to $B_1(\sigma) \times (-1, 1)$.
Due to \eqref{Stefanov_Wright_inverse} we need to consider
an expression of the form
\[
	J^{-2} G_{\varepsilon} ( v H_{\varepsilon} v )
	=
	G_0H_0 J^{-2} v^2
	+
	H_0 J^{-2} ( G_{\varepsilon} - G_0 ) v^2
	+
	G_{\varepsilon}
	\left( J^{-2} v J^2  \right)
	J^{-2} \left( H_{\varepsilon} - H_0 \right) v
	.
\]
Applying Lemma \ref{main_lemma}
and Corollary \ref{Selberg_Tesfahun_corollary_1}
we obtain
\[
	J^{-2} G_{\varepsilon} ( v H_{\varepsilon} v )
	=
	G_0H_0 J^{-2} v^2
	+
	O_{H^s}
	\left(
		| \varepsilon |^{ \frac {3 - 2\beta}{2 - \beta} }
	\right)
	.
\]
Similarly, by Lemma \ref{main_lemma} we estimate
\[
	J^{-2} F_{\varepsilon} v^2
	=
	F_0 J^{-2} v^2
	+
	O_{H^s}
	\left(
		| \varepsilon |^{ \frac {3 - 2\beta}{2 - \beta} }
	\right)
	.
\]
Finally, making use of Corollary \ref{main_corollary}
we obtain
\[
	J^{-2}
	\left( \omega_{\varepsilon} + M_{\varepsilon} \right)^{-1}
	\left[
		\omega_{\varepsilon} F_{\varepsilon} v^2
		+
		\omega_{\varepsilon}  G_{\varepsilon} ( v H_{\varepsilon} v )
	\right]
	=
	\frac 12
	\left[ F_0 + G_0H_0 \right]
	J^{-2} v^2
	+
	O_{H^s}
	\left(
		| \varepsilon |^{ \frac {3 - 2\beta}{2 - \beta} }
	\right)
\]
which together with \eqref{Phi_definition}
and  \eqref{Stefanov_Wright_inverse}
lead, as we show below, to
\begin{equation}
\label{Phi_around_zero}
	\Phi(v, \varepsilon)
	=
	v - \gamma J^{-2} v^2
	+
	o_{H^s} ( \varepsilon )
	+
	O_{H^s}
	\left(
		| \varepsilon |^{ 2 - \max(s_h, s_t) }  \norm{ v - \sigma }_{H^s}
	\right)
	.
\end{equation}
Indeed,
as one can easily see
\(
	\norm{H_{\varepsilon} \sigma}_{H^s}
	\leqslant
	C \norm{J_{\varepsilon}^{s_h} \sigma}_{H^s}
	\leqslant
	C \norm{\sigma}_{H^{s + s_h}}
	,
\)
and so
%
\[
	\varepsilon^2 G_{\varepsilon} ( v H_{\varepsilon} v )
	=
	\varepsilon^2 G_{\varepsilon} ( v H_{\varepsilon} \sigma )
	+
	\varepsilon^2 G_{\varepsilon} ( v H_{\varepsilon} (v - \sigma) )
	=
	O_{H^s}
	\left(
		\varepsilon^2
		+
		| \varepsilon |^{ 2 - s_h } \norm{ v - \sigma }_{H^s}
	\right)
	.
\]
Appealing to Assumption \ref{T_assumption} one obtains
\[
	\varepsilon^2 T_{\varepsilon} ( v, v, v )
	=
	\varepsilon^2 T_{\varepsilon} ( \sigma, \sigma, \sigma )
	+
	O_{H^s}
	\left(
		| \varepsilon |^{ 2 - s_t } \norm{ v - \sigma }_{H^s}
	\right)
	=
	o_{H^s}
	\left(
		\varepsilon
	\right)
	+
	O_{H^s}
	\left(
		| \varepsilon |^{ 2 - s_t } \norm{ v - \sigma }_{H^s}
	\right)
	,
\]
which finishes the proof of \eqref{Phi_around_zero}.

The first conclusion that we make from
Equality \eqref{Phi_around_zero} is that
$\Phi(v, 0)$ is well defined as a limit,
namely
\[
	\Phi(v, 0)
	=
	v - \gamma J^{-2} v^2
	,
\]
and in particular
\(
	\Phi(\sigma, 0) = 0
	.
\)
Secondly, we conclude that
\[
	\Phi(v, \varepsilon) =
	\mathcal K (v - \sigma)
	+
	O_{H^s}
	\left(
		\norm{ v - \sigma }_{H^s}^2
	\right)
	+
	o_{H^s} ( \varepsilon )
	+
	O_{H^s}
	\left(
		| \varepsilon |^{ 2 - \max(s_h, s_t) }  \norm{ v - \sigma }_{H^s}
	\right)
\]
with $\mathcal K$ defined in Lemma \ref{Stefanov_Wright_lemma_2},
which means in particular that
$\Phi$ is differentiable at the point $(\sigma, 0)$.

In order to appeal to the implicit function theorem
the following assumptions should be fulfilled
\begin{enumerate}
	\item
	$\Phi$ is continuous at the point $(\sigma, 0)$;
	\item
	$\Phi(\sigma, 0) = 0$;
	\item
	there exists $\partial_v \Phi$
	on the domain of $\Phi$ continuous at the point $(\sigma, 0)$;
	\item
	there exists
	\(
		( \partial_v \Phi(\sigma, 0) )^{-1}
		= \mathcal K^{-1}
	\)
	bounded in $H_e^s(\mathbb R)$.
\end{enumerate}
Note that the last assertion is proved in Lemma \ref{Stefanov_Wright_lemma_2},
and so it is left only to check the third statement.
A straightforward calculation for any $\varepsilon \ne 0$ gives
\begin{multline*}
	\partial_v \Phi(v, \varepsilon)w
	=
	w - \varepsilon^2 \left( \omega_{\varepsilon}^2 - M_{\varepsilon}^2 \right)^{-1}
	\left[
		2 \omega_{\varepsilon} F_{\varepsilon} (vw)
		+
		\omega_{\varepsilon}  G_{\varepsilon}
		( w H_{\varepsilon} v + v H_{\varepsilon} w )
	\right.	
	\\
	\left.
		+
		\varepsilon^2
		(
			T_{\varepsilon}(w,v,v)
			+
			T_{\varepsilon}(v,w,v)
			+
			T_{\varepsilon}(v,v,w)
		)
	\right]
\end{multline*}
and otherwise
\[
	\partial_v \Phi(v, 0)w
	=
	w - 2 \gamma J^{-2} (vw)
\]
for any $w \in H_e^s(\mathbb R)$.
Repeating the analysis given above one can obtain
\[
	\partial_v \Phi(v, \varepsilon) - \mathcal K
	=
	O_{ \mathcal B(H^s) }
	\left(
		\norm{ v - \sigma }_{H^s}
		+
		| \varepsilon |^{ \frac {3 - 2\beta}{2 - \beta} }
		+
		| \varepsilon |^{ 2 - \max(s_h, s_t) }
	\right)
\]
as
\(
	(v, \varepsilon) \to (\sigma, 0)
\)
in
\(
	H_e^s(\mathbb R) \times \mathbb R
	.
\)
Thus by the implicit function theorem
there exist $\varepsilon_0 > 0$
and a unique $V^{\varepsilon}$ defined for
\(
	\varepsilon \in ( - \varepsilon_0, \varepsilon_0 )
\)
and continuous at $\varepsilon = 0$ such that
\(
	\Phi( V^{\varepsilon}, \varepsilon ) = 0
\)
for each
\(
	\varepsilon \in ( - \varepsilon_0, \varepsilon_0 )
	.
\)

In order to obtain the asymptotic expression we
make use of \eqref{Phi_around_zero} again to deduce
\[
	V^{\varepsilon} - \sigma
	-
	\gamma \mathcal K^{-1} J^{-2} ( V^{\varepsilon} - \sigma )^2
	=
	o_{H^s}(\varepsilon)
	+
	O_{H^s}
	\left(
		| \varepsilon |^{ 2 - \max(s_h, s_t) }  \norm{ V^{\varepsilon} - \sigma }_{H^s}
	\right)
\]
and so
\[
	V^{\varepsilon} - \sigma
	=
	o_{H^s}(\varepsilon)
	+
	O_{H^s}
	\left(
		| \varepsilon |^{ 2 - \max(s_h, s_t) }  \norm{ V^{\varepsilon} - \sigma }_{H^s}
	\right)
	,
\]
which proves the asymptotic expression
and concludes the proof of Theorem \ref{main_theorem}.


\section{Conclusion}
\setcounter{equation}{0}

\subsection{(a,b,c,d)-Boussinesq systems}

Setting
\(
	K_a = 1 + a \partial_x^2
	,
\)
\(
	K_b = 1 - b \partial_x^2
	,
\)
\(
	K_c = 1 + c \partial_x^2
	,
\)
\(
	K_d = 1 - d \partial_x^2
\)
in \eqref{Regularised_sys},
where
\(
	a, b, c, d \in \mathbb R
	,
\)
we arrive to a system derived in \cite{Bona_Chen_Saut1}.
Its Cauchy problem was studied in \cite{Bona_Chen_Saut2}.
It exhibits solitary wave solutions,
as was shown in
\cite{Chen_Nguyen_Sun2010, Chen_Nguyen_Sun2011}
by the variational method.
Using Theorem \ref{main_theorem} we can extend significantly
their result at least for small surface tension,
namely, to a not necessary Hamiltonian case
corresponding to $b = d$.
Indeed,
let us shorten the range of coefficients $a, b, c, d$
in a way that System \eqref{Regularised_sys} is well posed
(see \cite{Bona_Chen_Saut2})
and $K_c$ is invertible, so that we obtain the first restriction
\begin{equation}
\label{abcd_condition_0}
	b, d \geqslant 0
	, \quad
	a,c \leqslant 0
	.
\end{equation}
Analyzing the symbol $M$ around zero
\[
	M(\xi)
	=
	\sqrt{
		\frac{( 1 - a\xi^2 )( 1 - c\xi^2 )}{( 1 + b\xi^2 )( 1 + d\xi^2 )}
	}
	=
	1 - \frac 12 (a + b + c + d)\xi^2 + O \left( \xi^4 \right)
\]
and taking into account \eqref{M_assumption_inequality_2}
we arrive to the second necessary restriction
\begin{equation}
\label{abcd_condition_1}
	a + b + c + d > 0
	.
\end{equation}
Condition \eqref{M_assumption_inequality_1} implies either
\begin{equation}
\label{abcd_condition_2}
	bd > ac
	\, \mbox{ or } \,
	bd = ac = 0
	.
\end{equation}
Finally, the fact that the loss of derivative $s_h < 2$,
associated with
\(
	H(\xi) = ( 1 + d\xi^2 ) / ( 1 - c\xi^2 )
	,
\)
leads to either
\begin{equation}
\label{abcd_condition_3}
	c < 0
	\, \mbox{ or } \,
	d = 0
	.
\end{equation}
One claims that Conditions
\eqref{abcd_condition_0}-\eqref{abcd_condition_3}
are sufficient to fulfil both
Assumptions \ref{main_assumption}, \ref{T_assumption},
and so to prove the existence and uniqueness of
small solitary waves.
Indeed,
\[
	T_{\varepsilon}(f, g, h)
	=
	\frac 12 G_{\varepsilon}
	\left(
		f \left( 1 - c \varepsilon^2 D^2\right)^{-1} (gh)
	\right)
	,
\]
where
\(
	G_{\varepsilon}
	=
	\left( 1 + b \varepsilon^2 D^2\right)^{-1}
	\left( 1 - c \varepsilon^2 D^2\right)
	\left( 1 + d \varepsilon^2 D^2\right)^{-1}
\)
is uniformly bounded due to
\eqref{abcd_condition_0}, \eqref{abcd_condition_1}
and so
\(
	\norm{ T_{\varepsilon} } \lesssim 1
	,
\)
which implies Assumption \ref{T_assumption}.
The rest is obvious.

\subsection{Aceves-S\'anchez-Minzoni-Panayotaros model}
Let us introduce a notation
\begin{equation}
\label{K_multiplier}
	K = \frac{\tanh D}{D}
\end{equation}
that we use in subsequent.
Assigning $K_a = K$
and
\(
	K_b = K_c = K_d = 1
\)
we arrive to a system introduced in
\cite{Aceves_Sanchez_Minzoni_Panayotaros}.
Its Cauchy problem was studied in \cite{Pei_Wang}.
It exhibits solitary wave solutions,
as was shown in
\cite{Nilsson_Wang}
by the variational method.
Note that
\(
	M = \sqrt{K}
	,
\)
\(
	F = 1/2
	,
\)
\(
	G = H = 1
	,
\)
and
\(
	T(f, g, h) = fgh/2
	.
\)
Thus we immediately obtain solitary waves
by Theorem \ref{main_theorem}.

\subsection{Hur-Pandey model}

Assigning $K_c = K$
and
\(
	K_a = K_b = K_d = 1
\)
we arrive to a system introduced in
\cite{Hur_Pandey}.
Note that
\(
	M = \sqrt{K}
	,
\)
\(
	F = 1/2
	,
\)
\(
	G = K
	,
\)
\(
	H = K^{-1}
	,
\)
and
\(
	T(f, g, h) = K \left( f K^{-1}(gh) \right) /2
	.
\)
Assumption \ref{main_assumption} is straightforward to check.
In order to proceed
note that
\[
	T_{\varepsilon}(f, g, h)
	=
	\frac 12 G_{\varepsilon}
	\left(
		f H_{\varepsilon} (gh)
	\right)
\]
and so appealing to
Corollary \ref{Selberg_Tesfahun_corollary_2}
one obtains Assumption \ref{T_assumption}
with $s_t = s_h = 1$.
Thus Theorem \ref{main_theorem} is applicable again.

\subsection{Dinvay-Dutykh-Kalisch model}

Setting eventually $K_c = 1$
and
\(
	K_a = K_b = K_d = K^{-1}
\)
we arrive to a system firstly
appeared in
\cite{Dinvay_Dutykh_Kalisch}.
It was proved later to be globally well-posed in
\cite{Dinvay_Tesfahun}.
It also admits existence of solitary waves
\cite{Dinvay_Nilsson} by the variational method.
It demonstrates a good performance even in simulations
of big waves \cite{Dinvay, Dinvay_Dutykh_Kalisch}.
Let us show that Theorem \ref{main_theorem} is applicable.
We have
\(
	M = \sqrt{K}
	,
\)
\(
	F = K/2
	,
\)
\(
	G = K^2
	,
\)
\(
	H = K^{-1}
	,
\)
and
\(
	T(f, g, h) = K^2 (fgh) /2
	.
\)
Assumptions \ref{main_assumption} and \ref{T_assumption}
are obviously satisfied.

As a final remark let us point out that
Theorem \ref{main_theorem} can be applied to
a more general Whitham-Boussinesq type system
as one introduced in Remark 5.3 of \cite{Emerald2020}.
It is possible to show that a linearisation
of System \eqref{Regularised_sys} around
a solitary wave solution leads to
an operator having both positive and negative essential spectrum.
This is a serious difficulty, so we leave investigation
of stability to a future work.


\vskip 0.05in
\noindent
{\bf Acknowledgments.}
{
	The author acknowledges the support of the ERC EU project 856408-STUOD.
}



\bibliographystyle{acm}
\bibliography{bibliography}

\end{document}